\documentclass[11pt]{article}
\usepackage{fullpage}
\usepackage{amsfonts, amsmath, amssymb,latexsym,epsfig}
\usepackage{amsthm}
\usepackage{tikz}
\usetikzlibrary{graphs}
\usetikzlibrary{graphs.standard}
\usepackage{relsize}
\usepackage{verbatim}
\usepackage{enumerate}
\usepackage[skip=10pt plus1pt, indent=20pt]{parskip}

\newtheorem{thm}{Theorem}[section]

\newtheorem{lem}[thm]{Lemma}
\newtheorem{remark}[thm]{Remark}

\newtheorem{conj}[thm]{Conjecture}

\newcommand{\F}{\mathbb F}

\newcommand{\hs}{\hspace{1em}}

\author{

Vladislav Taranchuk\footnotemark[1]}
\title{A simple proof for the lower bound of the girth of graphs $D(n, q)$}

\begin{document}

\maketitle
\footnotetext[1]{Department of Mathematical Sciences, University of Delaware, Newark, DE 19716-2553, USA, {\tt vladtar@udel.edu}.}

\begin{abstract}
    The components of the graphs $D(n, q)$ provide the best-known general lower bound for the number of edges in a graph with $n$ vertices and no cycles of length less than $g$. In this paper, we give a new, short, and simpler proof of the fact that the length of the shortest cycle appearing in $D(n, q)$ is $n + 5$ when $n$ is odd, and $n + 4$ when $n$ is even. %Consequently, we use similar ideas to construct many new families of graphs whose girth and other properties are worth investigating.
\end{abstract}

\section{Introduction and Motivation}\label{Section_intro}

Let $\Gamma$ be a graph with vertex set $V(\Gamma)$ and edge set $E(\Gamma)$. We use standard terminology from graph theory, see, e.g., Bollob\'{a}s \cite{B}. We refer to $|V(\Gamma)|$ and $|E(\Gamma)|$ as the order and the size of $\Gamma$, respectively. A graph is called $q$-regular if every vertex has exactly $q$ neighbors. All the graphs we consider are $q$-regular and simple (i.e., undirected, without loops and multiple edges). The girth of a graph containing at least one cycle is the length of the shortest cycle in the graph.

The graphs $D(n, q)$ are part of a broad class of graphs referred to in the literature as as \textit{algebraically defined graphs}, or \textit{graphs defined by systems of equations}, see Lazebnik and Woldar \cite{LW},  Lazebnik, Sun, and Wang \cite{LSW}. The graphs in this class are defined as follows: Let $n \geq 2$ and $P= L = \mathbb{F}_q^{n}$ be two copies of the $n$-dimensional vector space over the finite field $\mathbb{F}_q$ of order $q = p^e$, where $p$ is prime. Call the elements of $P$ points and the elements of $L$ lines, with the distinction in notation: For $a \in \F_q^n$, $(a) \in P$ and $[a] \in L$. The point $(0, 0, \dots, 0)$ will be denoted by $(0)$.
For each $i$, $2 \leq i \leq n$, let $f_i$ be a polynomial of $2i - 2$ variables over $\F_q$.
Define $\Gamma_q = \Gamma_q(f_2, f_3, \dots, f_n)$ to be the bipartite graph with parts $P$ and $L$ and with edge relation defined as follows: If $(p) = (p_1, \dots, p_{n})\in P$ and $[\ell] = [\ell_1, \dots, \ell_{n}]$, then $(p)$ is adjacent to $[l]$ if and only if 
\begin{align*}
\ell_2 + p_2 &= f_2(\ell_1, p_1), \\
\ell_3 + p_3 &= f_3(\ell_1, p_1, \ell_2, p_2), \\
\vdots \\
\ell_{n} + p_{n} &= f_n(\ell_1, p_1, \dots, \ell_{n-1}, p_{n-1}).
\end{align*}

\noindent This class of graphs has connections and applications to finite geometry, extremal graph theory, cryptography, and coding theory \cite{U05, LW, LSW}.

The graphs $D(n, q)$ are a family of algebraically defined graphs that were constructed by Lazebnik and Ustimenko in 1994 \cite{LU}. These graphs are defined by a system of $n - 1$ equations which give the following adjacency relation between a point $(p) = (p_1, \dots, p_n)$ and a line $[\ell] = [\ell_1, \dots, \ell_n]$:
\begin{align*}
    &p_2 + \ell_2 = p_1\ell_1 \\
    &p_3 + \ell_3  = p_1\ell_2 \\
    &\text{and for }4 \leq j \leq n: \\
    &p_{j} + \ell_{j}  = \left\{  \begin{array}{ll}
      p_{j-2}\ell_1   & \mbox{if $j \equiv 0, 1\pmod{4}$} \\
      p_1\ell_{j-2}   & \mbox{if $j \equiv 2, 3\pmod{4}$}
    \end{array}
    \right. 
\end{align*}
\noindent For $n \geq 6$, the graphs $D(n, q)$ are known to be disconnected and their components are completely described, see Lazebnik, Ustimenko, and Woldar \cite{LUW2}, Lazebnik and Viglione \cite{LV}. Furthermore, $D(n, q)$ is known to be transitive on its set of points and set of lines \cite{LU}. In the same paper, the graphs $D(n, q)$ were shown to have girth at least $n + 5$ when $n$ is odd, and $n+4$ when $n$ is even. These bounds have been show to be tight in many cases, see F\"{u}redi, et. al, \cite{FLSUW}, Cheng, Chen, and Tang \cite{CCT1, CCT2}, Cheng, Tang, and Xu \cite{CTX}. The original proof of the lower bound for the girth in \cite{LU} was rather technical, and it was presented in a somewhat simpler way with new notation in \cite{LSW}. In this paper we will provide a shorter and simpler proof of this fact, which was previously attempted by Ustimenko \cite{UP}.

Let $\mathcal{F}$ be a set of graphs. The \textit{Tur\'{a}n number} of $\mathcal{F}$, denoted $ex(n, \mathcal{F})$, is the maximum number of edges that a graph on $n$ vertices can have such that it does not contain any graph in $\mathcal{F}$ as a subgraph. Let $\mathcal{C}_k$ be the set of all cycles with length less than or equal to $k$. One of the most notorious open questions in graph theory asks to determine the functions $ex(n, \mathcal{C}_{k})$ (i.e. maximum number of edges in a graph on $n$ vertices with girth at least $k$). To date, the best-known general bounds are
\begin{align}\label{Eq1}
   \frac{1}{2^{1 + 1/k}}n^{1 + 2/(3k-3 +\epsilon)} \leq ex(n, \mathcal{C}_{2k+1}) \leq \frac{1}{2^{1 + 1/k}}n^{1 + 1/k} + \frac{1}{2}n
\end{align}
where $\epsilon = 1$ if $k$ is even and $0$ if $k$ is odd. An upper bound of the same magnitude was originally established almost 50 years ago by Bondy and Simonovits \cite{BS}. The upper bound in (\ref{Eq1}) is due to Alon, Hoory, and Linial \cite{AHL}. The lower bound is obtained by the components of the graphs $D(n, q)$ which are denoted $CD(n, q)$. This discovery lead to an improvement over the previously best-known lower bound given by Lubotzky, Philips, and Sarnak \cite{LPS},  Margulis \cite{Ma}, and Morgenstern \cite{MOR}. 

When $k = 2, 3, 5$, more precise results are known regarding $ex(n, \mathcal{C}_{2k+1})$. In fact, the graphs with the maximal edge count are known for infinitely many values of $n$, and these graphs come from finite geometries, namely generalized polygons. For more history and information regarding $ex(n, \mathcal{C}_{k})$ and related problems, see F\"{u}redi and Simonovits \cite{FS}, \cite{LSW}.

%Define $\theta_{k, t}$ to be the graph which consists of $t$-edge disjoint paths of length $k$ connecting two vertices. In this manner, the cycle $C_{2k}$ can be thought of as $\theta_{2, k}$. The Tur\'{a}n numbers $ex(n, C_{2k})$ and $ex(n, \theta_{t, k})$ have also been well studied. The order of magn

%Before we proceed to the details, we will give a high level overview of the main ideas.

Before proceeding to the proof of the main results, we require a few more definitions and comments. A \textit{covering map} $\Phi$ from the graph $\Gamma$ onto a graph $\Gamma'$ is a surjection $\Phi: V(\Gamma) \rightarrow V(\Gamma')$ such that the neighborhood of each $x \in V(\Gamma)$ is mapped bijectively onto the neighborhood of $\Phi(x)$. If a covering map from $\Gamma$ to $\Gamma'$ exists, we say that $\Gamma$ is a \textit{lift} of $\Gamma'$. If $\Phi: \Gamma \rightarrow \Gamma'$ is a covering map, then it is immediate that:
\begin{itemize}
    \item If $C \subset \Gamma$ is a cycle, then $\Phi(C)\subset \Gamma'$ contains a cycle with length at most the length of $C$.
    \item The girth of $\Gamma$ is at least the girth of $\Gamma'$.
\end{itemize}
\begin{remark}
    Let $2 \leq m < n$ and consider algebraically defined graphs $\Gamma = \Gamma_q(f_2, f_3, \dots, f_n)$ and $\Gamma' = \Gamma_q(f_2, f_3, \dots, f_m)$. Then, it is known that $\Gamma$ is a lift of $\Gamma'$ with the covering map $\Phi((p_1, \dots, p_n)) = (p_1, \dots, p_m)$ and $\Phi([\ell_1, \dots, \ell_n])$ $= [\ell_1, \dots, \ell_m]$ \cite[Section 3.2, Theorem 2]{LSW}.
\end{remark}

Here we present in a nutshell the main ideas used in the proof of our main result. A new family of algebraically defined graphs, denoted $A(n, q)$, $n \geq 2$, was introduced by Ustimenko in \cite[p. 467]{U07}, where they were denoted by $E(n, k)$, and later in \cite{U13, U22}, where they were denoted by $A(n, q)$. Though the girth of $A(n, q)$ was not as was hoped in \cite{U22}, Proposition 3.1 in this preprint states that the point $(0)$ is not contained in a cycle of length less than $2n$. In fact, the key theorem of our paper improves upon this, demonstrating that the point $(0)$ in $A(n, q)$ is not contained in any cycle of length less than $2n + 2$ (which was also suggested in \cite{U13}). We then demonstrate that there is a covering map from $D(2k + 1, q)$ to $A(k+2, q)$ which maps the point $(0)$ in $D(2k + 1, q)$ to the point $(0)$ in $A(k + 2, q)$. Finally, by transitivity of $D(2k + 1, q)$ on the set of points and properties of covering maps, we immediately obtain that no cycle of length less than $2k + 6$ can appear in $D(2k + 1, q)$.

\section{The graphs $A(n,q)$ and $D(n, q)$}
The algebraically defined graph $A(n, q)$ has edge relation defined as follows: For $2 \leq j \leq n$,
\begin{align*}
    &p_{j} + \ell_{j}  = \left\{  \begin{array}{ll}
      p_{j - 1}\ell_1   & \mbox{if $j$ is even} \\
      p_1\ell_{j - 1}   & \mbox{if $j$ is odd}
    \end{array}
    \right.
\end{align*}
The following lemma, though not explicitly, appears in \cite[p. 467-468]{U07}. For our purposes, we make the lemma explicit and provide a proof.
\begin{lem}
    Let $k \geq 1$ and $\phi: \F_q^{2k + 1} \rightarrow \F_q^{k + 2}$ where $\phi(x) = x'$ is defined coordinate-wise as follows:
    \begin{align*}
    x'_1 &= x_1 \\
    x'_2 &= x_2   \\
    x'_{j + 2} &= x_{2j + 1} \hspace{1em} \text{for } 1 \leq j \leq k
\end{align*}
Then $\Phi: D(2k + 1, q) \rightarrow A(k + 2, q)$, defined by $\Phi((p)) = (\phi(p))$ and $\Phi([\ell]) = [\phi(\ell)]$, is a covering map.
\end{lem}

\begin{proof}  
    It is clear that $\Phi$ is a surjection.
    Suppose that $(p)$ is adjacent to $[\ell]$ in $D(2k + 1, q)$. Then it must be true that 
    \begin{align*}
    p_2 + \ell_2 &= p_1\ell_1 \\
    p_3 + \ell_3 & = p_1\ell_2 \\
    \text{and for } &2 \leq j \leq k:\\
    p_{2j+ 1} + \ell_{2j + 1} & = \left\{  \begin{array}{ll}
      p_{2j - 1}\ell_1   & \mbox{if $j$ is even} \\
      p_1\ell_{2j - 1}   & \mbox{if $j$ is odd}
    \end{array}
    \right.
\end{align*}
But then applying $\Phi$, where $\Phi((p)) = (p')$ and $\Phi([\ell]) = [\ell']$, we obtain 
\begin{align*}
    p_2' + \ell_2' = p_2 + \ell_2 &= p_1\ell_1  = p_1'\ell_1'\\
    p_3' + \ell_3' = p_3 + \ell_3 & = p_1\ell_2 = p_1'\ell_2' \\
    \text{and for } 2 \leq j \leq k& \\
    p_{j+2}' + \ell_{j + 2}'  = p_{2j+ 1} + \ell_{2j + 1} & = \left\{  \begin{array}{ll}
      p_{2j - 1}\ell_1   & \mbox{if $j$ is even} \\
      p_1\ell_{2j - 1}   & \mbox{if $j$ is odd}
    \end{array}
    \right\} = \left\{  \begin{array}{ll}
      p_{j + 1}'\ell_1'   & \mbox{if $j$ is even} \\
      p_1'\ell_{j+1}'   & \mbox{if $j$ is odd}
    \end{array}
    \right.
\end{align*}
Thus $\Phi((p))$ is adjacent to $\Phi([\ell])$, so $\Phi$ preserves adjacency. If $v$ is a vertex in an any algebraically defined graph constructed over $\F_q$, then for each $x \in \F_q$, $v$ has a unqiue neighbor whose first coordinate is $x$ \cite[Theorem 1]{LSW}. Since the first coordinate of $v \in \F_q^{2k + 1}$ and $\phi(v) \in \F_q^{k+2}$ are the same, then for any vertex $v$ of $D(2k  + 1, q)$, we have that $\Phi$ must be a bijection from the neighborhood of $v$ to the neighborhood of $\Phi(v)$. Thus $\Phi$ is a covering map.
\end{proof}

\begin{lem}
    Let $n \geq 2$ and $(0)$ be in $A(n, q)$. Then a vertex at distance $j$, $2 \leq j \leq n$, from $(0)$ is of the form:
    \begin{align}
    &[\ell_1, \ell_2, \dots, \ell_{j}, 0, \dots, 0], \hspace{1em} \ell_{j-1} \neq 0 \hspace{1cm} \text{if $j$ is odd}, \\
    &(p_1, p_2, \dots, p_{j}, 0, \dots, 0), \hspace{1em} p_{j -1} \neq 0 \hspace{1cm} \text{if $j$ is even}.
\end{align}
\end{lem}

\begin{proof}
    For a fixed $n$, we proceed by induction on $j$. Note that the base case ($j = 2$) is true since 
    $$
    (0) \sim [a, 0, \dots, 0] \sim (x, ax, 0, \dots, 0) \hs a, x \in \F_q, x \neq 0
    $$
    Thus, we may also assume $n \geq 3$. Suppose that for a fixed $n \geq 3$, the lemma holds up to $j \leq n-1$. If $j$ is odd, then by the inductive hypothesis, we have a point $(q) = (q_1, \dots, q_{j-1}, 0, \dots)$ at distance $j-1$ from $(0)$ adjacent to a line $[\ell] = [\ell_1, \dots, \ell_j, 0, \dots]$ at distance $j$ from $(0)$ satisfying $\ell_{j- 1} \neq 0$ and $\ell_{j} = q_1\ell_{j-1}$. Next, we note that if $(p) = (p_1, \dots, p_n)$ is another neighbor of $[\ell_1, \dots, \ell_j, 0, \dots]$, then since $\ell_m = 0$ for $j < m \leq n$ we have
    \begin{align*}
        p_j + \ell_j &= p_1\ell_{j-1} \\
        p_{j + 1}  &= p_j\ell_1 \\
        p_{j + 2} &= p_1\ell_{j+1} = 0  \hspace{1em} (\text{since $\ell_{j + 1} = 0$})\\
        p_{j + 3} &= p_{j+2}\ell_1 = 0 \hspace{1em} (\text{since $p_{j + 2} = 0$})\\
        &\vdots
    \end{align*}
    Thus $p_m = 0$ for $m > j + 1$. Furthermore, using the fact that $\ell_{j} = q_1\ell_{j-1}$ we obtain
    $$
    p_{j} = p_1\ell_{j-1} - \ell_j = (p_1 - q_1)\ell_{j-1}.
    $$
    Since $\ell_{j-1} \neq 0$ (by the inductive hypothesis) and $p_1 \neq q_1$ since both $(p)$ and $(q)$ are distinct neighbors of $[\ell]$, then $p_{j} \neq 0$ as claimed. Furthermore, $p_j \neq 0$ implies that $(p)$ cannot be at distance $j -1$ as it would contradict the inductive hypothesis, and so  $(p)$ must be at distance $j  + 1$.
    
    Now, if we instead begin with $j$ even, we can swap the roles of points and lines and perform a similar argument to prove this case. Together, these two cases finish the proof of the lemma.
\end{proof}

\begin{thm}
    For $n \geq 2$, the length of the shortest cycle in $A(n, q)$ containing the point $(0)$ is at least $2n + 2$.
\end{thm}

\begin{proof}
    We proceed by induction on $n$. For $n = 2, 3$, the graphs $A(n, q)$ are the same as the graphs $D(2, q)$ and $D(3, q)$ whose girths are known to be $6$ and $8$ respectively. Therefore, beginning at $n = 2$, the base case is done. 

    Suppose that no cycle of length less than $2n + 2$ in $A(n, q)$ contains the point $(0)$. Using the covering map from $A(n+1, q)$ to $A(n, q)$ which ``deletes" the $(n+1)^{\text{st}}$ coordinate and $n^{\text{th}}$ equation, we immediately obtain that $A(n + 1, q)$ has no cycles of length less than $2n + 2$ through the point $(0)$. All that is left is to demonstrate that no cycle of length $2n + 2$ in $A(n + 1, q)$ contains the point $(0)$. 

    If $n + 1$ is even, then $n$ is odd, and by Lemma 2.2, we have that any line at distance $n$ from $(0)$ is of the form 
    \begin{align*}
        [\ell_1, \dots, \ell_{n}, 0], \hspace{1cm} \ell_{n - 1} \neq 0.
    \end{align*}
    A point at distance $n + 1$ is of the form 
    $$
    (p_1, \dots, p_{n+1}), \hspace{1cm} p_{n} \neq 0.
    $$
    Suppose that $A(n + 1, q)$ has a cycle of length $2n+2$ containing the point $(0)$. Then there must exist a point at distance $n + 1$ from $(0)$ with two common neighbors in the set of all lines at distance $n$ from $(0)$. Denote these lines $[\ell_1, \dots, \ell_{n}, 0]$ and $[k_1, \dots, k_{n}, 0]$. Then as $\ell_{n + 1} = k_{n + 1} = 0$, we have
    $$
    p_{n+ 1} = p_{n}\ell_1 = p_{n}k_1.
    $$
    As $p_{n} \neq 0$, it must be that $\ell_1 = k_1$, which contradicts the fact that the lines $[\ell]$ and $[k]$ are distinct neighbors of $(p)$.

    If $n+1$ is odd, then again as in the proof of Lemma 2.2, swapping the roles of points and lines, we perform a completely similar argument, finishing the proof of the theorem.
\end{proof}

\noindent We are now ready to prove the main result of this paper.

\begin{thm}
    The girth of $D(n, q)$ is at least $n + 5$ when $n$ is odd, and at least $n + 4$ when $n$ is even.
\end{thm}

\begin{proof}
    Let $n = 2k + 1$ and $k \geq 1$. It is known that $D(2k + 1, q)$ is transitive on the set of points and lines. Therefore, it contains a cycle of length $2k + 4$ or less, if and only if it contains a cycle of the same length through the point $(0)$. Lemma 2.1 gives a covering map $\Phi$ from $D(2k + 1, q)$ to $A(k +2, q)$ such that $\Phi((0)) = (0)$. Therefore, if a cycle of length $2k + 4$ or less exists in $D(2k + 1, q)$, then a cycle of the same length or less would exist in $A(k+2, q)$ and contain the point $(0)$, which contradicts Theorem 2.3. Thus the girth of $D(2k + 1, q)$ is at least $2k + 6 = n + 5$. If $n = 2k + 2$, then we know that the girth of $D(2k + 2, q)$ is greater than or equal to the girth of $D(2k + 1, q)$, which we have shown is at least $2k + 6 = n + 4$.
\end{proof}

As we have observed the equations governing the edge relations of $A(n,q)$ are in fact a subset of the equations governing the edge relations of $D(2n - 3, q)$. The remaining equations of $D(2n -3, q)$ are present for the purpose of making the graphs edge transitive (as well as point transitive and line transitive). It also turned out that some of these remaining equations cause the graphs $D(2n - 3, q)$ to disconnect, which was unexpected. The graphs $A(n, q)$ can in a sense be seen as the frame around which $D(2n-3, q)$ is built. It would be worthwhile to investigate other families of graphs that can serve as potential ``frames" around which new high girth families of graphs can be built.

\section{Conclusion and Open Questions}

It is our belief that researchers have only scratched the surface in regards to applying algebraically defined graphs and their generalizations to various problems in graph theory. We hope that this paper makes this area more accessible to those interested in studying such problems. We leave here several problems and conjectures of interest.

It has been observed computationally that the graphs $A(n, q)$ are connected for $q = 3, 4, 5, 7$ and $2 \leq n \leq 10$. It seems likely that this is always the case when $q > 2$. If this is not the case, their components may be of interest, just like it happened with the components $CD(n, q)$ of $D(n, q)$.

\begin{conj}
    The graphs $A(n, q)$ are connected for all prime powers $q > 2$ and all $n \geq 2$.
\end{conj}

\noindent The graphs $D(n, q)$ have been conjectured by Ustimenko to be nearly Ramanujan, that is their second largest eigenvalues is less than or equal to $2\sqrt{q}$. This conjecture has been verified for $k = 2, 3, 4$ for all prime powers $q$, and $k = 5$ for all odd prime powers $q$, see Li, Lu, and Wang \cite{LLW}, Cioab\u{a}, Lazebnik and Li \cite{CLL}, Moorhouse, Sun, and Williford \cite{MSW}, Gupta and Taranchuk \cite{GT}. Since $D(2k + 1, q)$ is lift of $A(k + 2, q)$, then the spectrum of $A(k +2, q)$ embeds into the spectrum of $D(2k + 1, q)$ \cite{LSW}. The graphs $A(2, q)$ and $A(3, q)$ are precisely the same graphs as $D(2, q)$ and $D(3, q)$, or the Wenger graphs $W_1(q)$ and $W_2(q)$ see \cite{CLL}, whose spectrum is known. $D(5, q)$ is a lift of $A(4, q)$, and therefore, for all odd prime powers $q$, the second largest eigenvalue of $A(4, q)$ is less than or equal to $2\sqrt{q}$ \cite{GT}. Ustimenko's conjecture regarding the second largest eigenvalue of $D(n, q)$ would imply the same for the family $A(n, q)$.

\begin{conj}
    The second largest eigenvalue of $A(n, q)$ is less than or equal to $2 \sqrt{q}$.
\end{conj}

\section{Acknowledgements}
We wish to thank Boris Bukh and Felix Lazebnik for bringing to our attention the preprint of Ustimenko \cite{U22}. The author enjoyed many useful discussions with Felix Lazebnik regarding algebraically defined graphs and appreciates his feedback on the previous versions of this paper. We also thank Grahame Erskine \cite{E} for independently verifying and extending our computer results regarding the girth of $A(n, q)$ for small values of $n$ and $q$. This work was partially supported by the Simons Foundation Award ID: 426092 and the National Science Foundation Grant: 1855723.


\begin{thebibliography}{99}
    %\bibitem{AD}
    %E. Abajo and A. D\'{i}anez, Graphs with maximum size and lower bounded girth, Applied
    %Mathematics Letters, 25(2012): 575–579.

    \bibitem{AHL} N. Alon, S. Hoory and N. Linial, {\em Moore bound for irregular graphs}, Graphs and Combinatorics, 18(1), (2002): 53–57.

    \bibitem {B} B. Bollob\'{a}s, {\em Modern Graph Theory}, Springer-Verlag New York Inc. (1998).

    \bibitem{BS}
     J. A. Bondy and M. Simonovits, {\em Cycles of even length in graphs}, Journal of Combinatorial Theory, Series B, 16, (1974): 97–105.

    \bibitem{CCT1}
    X. Cheng, W. Chen, and Y. Tang, {\em On the girth of the bipartite graph D(k, q)}, Discrete Mathematics, 335(2014): 25–34.

    \bibitem{CCT2}
    X. Cheng, W. Chen, and Y. Tang, {\em On the conjecture for the girth of the bipartite graph D(k, q)}, Discrete Mathematics, 339, (2016): 2384–2392.

    \bibitem{CTX}
    X. Cheng, Y. Tang, and M. Xu, {\em Girth of the algebraic bipartite graph $D(k,q)$}, {\tt https://arxiv.org/abs/2209.01896}, (2022).

    \bibitem{CLL}
    S. M. Cioab\u{a}, F. Lazebnik, and W. Li, {\em On the spectrum of Wenger graphs}, Journal of
    Combinatorial Theory Series B, 107, (2014): 132–139.

    \bibitem{E}
    G. Erskine, {\em Private Communication}, (December 2022).

    \bibitem{FLSUW}
    Z. F\"{u}redi, F. Lazebnik, A. Seress, V.A. Ustimenko, and A.J. Woldar, {\em Graphs of prescribed girth and bi-degree}, Journal of Combinatorial Theory, Series B, 64(2), (1995): 228–239.

    \bibitem{FS}
    Z. F\"{u}redi and  M. Simonovits, {\em  The History of Degenerate (Bipartite) Extremal Graph Problems},  Erd\"{o}s Centennial Bolyai Society Mathematical Studies, 25, (2013).


    \bibitem{GT}
    H. Gupta and V. Taranchuk, {\em On the eigenvalues of the graphs $D(5, q)$}, {\tt https://arxiv.org/abs/2207.04629}, (2022).

    %\bibitem{H}
    %S. Hoory, The size of bipartite graphs with a given girth, Journal of Combinatorial Theory, Series B, 86(2), (2002): 215–220.

    %\bibitem{LS}
    %F. Lazebnik and S. Sun, {\tt http://udel.edu/ %shuying/girthDkq.pdf}

    \bibitem{LSW}
    F. Lazebnik, S. Sun, and Y. Wang, {\em Some families of graphs, hypergraphs and digraphs
    defined by systems of equations: a survey}, Selected topics in graph theory and
    its applications: Lecture Notes in Seminal Interdisciplinary Mathematics, 14, (2017): 105–142.

    \bibitem {LU} F. Lazebnik and V. A. Ustimenko, {\em Explicit construction of graphs with an arbitrary large girth and of large size}, Discrete Applied Mathematics, 60, (1995): 275–284.

    

    \bibitem {LUW} F. Lazebnik, V. A. Ustimenko, and A. J. Woldar, {\em A new series of dense graphs of high girth}, Bulletin of the American Mathematical Society, 32(1), (1995): 73-79.
    
    \bibitem{LUW2}
    F. Lazebnik, V. A. Ustimenko, and A. J. Woldar{\em A Characterization of the Components of the graphs $D(k,q)$},
    Discrete Mathematics, 157, (1996): 271–283.

    \bibitem{LV}
    F. Lazebnik and R. Viglione, {\em On the connectivity of certain graphs of high girth}, Discrete Mathematics, 277, (2004): 309–319.
    
    \bibitem{LW}
    F. Lazebnik and A.J. Woldar, {\em General properties of some families of graphs defined
    by systems of equations}, Journal of Graph Theory, 38 (2001): 65-86.

     \bibitem{LLW}
    W.-C. W. Li, M. Lu, and C. Wang, {\em Recent developments in low-density parity-check codes},
    Coding and cryptology, 107–123, Lecture Notes in Computer Science, 5557,
    (2009).

    \bibitem{LPS} A. Lubotzky, R. Phillips, and P. Sarnak, {\em Ramanujan graphs}, Combinatorica, 8, (1988): 261-277.

    \bibitem {Ma} G. A. Margulis, {\em Explicit group-theoretical constructions of combinatorial schemes and their application to the design of expanders and concentrators}, Journal of Probability and Information Transmission, 24(1), (1988): 39–46.

    \bibitem{MOR} M. Morgenstern, {\em Existence and explicit constructions of q + 1 regular Ramanujan graphs for every prime power q},  Journal of Combinatorial Theory Series B, 62, (1994): 44-62.

    \bibitem{MSW} G. E. Moorhouse, S. Sun, and  J. Williford, {\em The Eigenvalues of the Graphs $D(4, q)$}, Journal of Combinatorial Theory Series B, 125, (2017): 1–20.

    %\bibitem{N}
    %I. Neuwirth, {\em The size of bipartite graphs with girth eight}, {\tt arXiv:math.CO/0102210,} (2001).
    \bibitem{U05}
    V. A. Ustimenko, {\em Maximality of affine group and hidden graph cryptsystems}, Journal of Algebra and Discrete Mathematics, October, 10, (2004): 51-65. 

    \bibitem{U07}
    V. A. Ustimenko, {\em On linguistic dynamical systems, families of graphs of large girth, and cryptography}, Journal of Mathematical Sciences, 140, (2007), 461-471.

    \bibitem{U13}
    V. A. Ustimenko, {\em On the extremal graph theory and symbolic computations}, Reports
    of the National Academy of Sciences of Ukraine, 2, (2013): 42-49. (\textit{in Russian})

    \bibitem{U22}
    V. A. Ustimenko, {\em New results on algebraic graphs of large girth and their impact on Extremal Graph Theory and Algebraic Cryptography}, {\tt https://eprint.iacr.org/2022/1489.pdf}, (2022).

    \bibitem{UP}
    V. A. Ustimenko, {\em Personal Communication}, (November 2022).
    
\end{thebibliography}
\end{document}